\documentclass{math-note}

\usepackage[english]{babel}
\usepackage{mathtools,amssymb,amsthm}
\usepackage{xurl}
\usepackage[
  hidelinks,
  pdftitle={An Improved Lower Bound for the Complex Grothendieck Constant},
  pdfauthor={Shengtao Guo, Ethan X. Fang, Junwei Lu},
  pdfsubject={An Improved Lower Bound for the Complex Grothendieck Constant},
  pdfkeywords={Complex Grothendieck constant, complex Wiener chaos, Hermite multiplier, computer-assisted proof, interval arithmetic}
]{hyperref}

\title{\Large An Improved Lower Bound for the Complex Grothendieck Constant} 

\author{
  Shengtao Guo \qquad
  Ethan X. Fang \qquad
  Junwei Lu\thanks{Department of Biostatistics, Harvard T.H. Chan School of
    Public Health. Email: \texttt{junweilu@hsph.harvard.edu}.}
}

\newtheorem{theorem}{Theorem}[section]
\newtheorem{proposition}[theorem]{Proposition}
\newtheorem{lemma}[theorem]{Lemma}

\numberwithin{equation}{section}
\theoremstyle{definition}
\newtheorem{definition}[theorem]{Definition}
\theoremstyle{remark}

\DeclareMathOperator{\RePart}{Re}
\DeclarePairedDelimiter{\abs}{\lvert}{\rvert}
\DeclarePairedDelimiter{\norm}{\lVert}{\rVert}
\newcommand{\C}{\mathbb C}

\newcommand{\E}{\mathbb E}
\newcommand{\KG}{K_G^{\mathbb C}}
\newcommand{\Kstar}{\mathcal K_*}
\newcommand{\dd}{\,\mathrm d}

\begin{document}

\maketitle

\begin{abstract}
We prove \(K_G^{\mathbb C}>1.35584631827168\) for the classical complex
Grothendieck constant, closing more than one quarter of the gap between
Davie's lower bound~\cite{FriedlandLimZhang2018} and Haagerup's upper
bound~\cite{Haagerup1987}.
The numerical part of the proof is rigorously verified by interval
arithmetic.
The lower bound and the proof were discovered by the Odin Automatic AI
Research Agent.
\end{abstract}

\section{Introduction}

The classical complex Grothendieck constant \(\KG\) is the least number \(K\)
such that
\begin{equation*}
\abs*{\sum_{i=1}^m\sum_{j=1}^n a_{ij}\langle x_i,y_j\rangle}
\le
K\max_{\abs{\varepsilon_i}=\abs{\delta_j}=1}
\abs*{\sum_{i=1}^m\sum_{j=1}^n a_{ij}\varepsilon_i\overline{\delta_j}}
\end{equation*}
for all positive integers \(m,n,d\), every complex matrix
\(A=(a_{ij})\in\C^{m\times n}\), and all unit vectors \(x_i,y_j\in\C^d\).
Throughout, complex inner products are linear in the first variable.
Thus \(\KG\) is the optimal factor for comparing the vector-valued
maximum with the maximum over complex scalars of modulus one.

Grothendieck proved that \(\KG\) is finite~\cite{Grothendieck1953}.
Its exact value remains unknown~\cite{Pisier2012}.
Davie's lower bound is approximately \(1.33807\), as recorded in
\cite[Section~1(v), p.~329]{FriedlandLimZhang2018}.
Haagerup's upper bound is approximately
\(1.40491\)~\cite[Theorem~3.1]{Haagerup1987}.

\begin{theorem}\label{thm:main-lower-bound}
The classical complex Grothendieck constant satisfies
\[
\KG>1.35584631827168.
\]
\end{theorem}

\medskip\noindent\emph{Related work.}
The Davie--Reeds construction uses a Gaussian projection minus a multiple
of the identity~\cite[pp.~2--3]{Reeds1991}.
Heilman and, in concurrent work, Jones and Malavolta improved the lower bound
for the real Grothendieck constant by adding a cubic Hermite
correction~\cite{Heilman2026,JonesMalavolta2026}.
Saha et al.~\cite{SahaEtAl2026} obtained a further lower bound for the real
constant.
For the complex problem, Heilman proposed the analogous higher-order Hermite
perturbation~\cite[Section~1.8]{Heilman2026}.
We implement this proposal using finitely many additional Hermite
projections and a common radial weight that controls their contributions
jointly.
Hermite perturbations also occur in the work of Braverman, Makarychev,
Makarychev and Naor~\cite[Sections~4--5]{BravermanEtAl2013}; there they
modify a Gaussian rounding function to improve the real upper bound.

\medskip\noindent\emph{Proof strategy.}
We compare the limiting vector-valued Gaussian pairing with a
dimension-independent upper bound for the scalar
\(L_\infty\)-to-\(L_1\) operator norm.
The operator formulation \cite[Theorem~2.5]{Pisier2012} then gives a
lower bound for \(\KG\).

On \(\C^n\) with standard circular complex Gaussian measure, write
\(Q_\ell\) for the projection onto Hermite bidegree
\((\ell+1,\ell)\), and set \(\Pi_1=\sum_{\ell\ge0}Q_\ell\).
We use the multipliers
\[
T_n=Q_0-\rho\Pi_1-\sum_{\ell\in\Lambda}\beta_\ell Q_\ell,
\qquad
0<\rho<1,\quad \beta_\ell>0,
\]
where \(\Lambda\subset\mathbb N_{\ge1}\) is finite and the parameters do
not depend on \(n\).
The vector-valued test function \(F_n(z)=z/\norm z_2\), with
\(F_n(0)=0\), satisfies
\(\langle T_nF_n,F_n\rangle\to1-\rho\) as \(n\to\infty\).
Indeed, its \(L_2\)-distance to \(z/\sqrt n\) tends to zero, and each
additional Hermite projection annihilates this linear test.

The main difficulty is to bound \(\norm{T_n}_{\infty\to1}\) uniformly in
\(n\). After separating one Gaussian coordinate, fix a Hermite basis
element in the remaining coordinates.
Several active levels contribute moments of the same radial profile;
we estimate these moments jointly.
A common positive weight \(W\) controls them through finite Gram
matrices, retaining the correlations between levels.
The remaining quadratic-form inequality for radial functions reduces to
one scalar Schur-complement condition.
Writing \(\mu_W=\int_0^\infty W(s)e^{-s}\dd s\), the resulting criterion
in Theorem~\ref{thm:weighted-chaos} gives
\(\norm{T_n}_{\infty\to1}\le\rho+\mu_W\), and hence
\(\KG\ge(1-\rho)/(\rho+\mu_W)\).
Section~\ref{sec:certificate} supplies exact rational parameters on
\(Q_1,\ldots,Q_{10}\), of total Hermite degrees \(3,5,\ldots,21\), and a
two-pole rational weight.
Their interval verification proves Theorem~\ref{thm:main-lower-bound}.

We also bound the possible improvement within this weighted criterion.
Let \(\Kstar\) be the supremum of \((1-\rho)/(\rho+\mu_W)\) over all
parameters satisfying its Gram and Schur conditions, allowing both the
finite active set and the positive weight to vary.
Section~\ref{sec:parameter-optimization} proves
\[
1.35584631827168<\Kstar<1.35584697425050.
\]
The constructed ratio is therefore within \(6.56\times10^{-7}\) of the
best value supplied by these conditions.
The upper endpoint concerns \(\Kstar\), not \(\KG\).

\medskip\noindent\textbf{The role of AI in this proof.}
Odin Automatic AI Research Agent was used to derive the lower bound and the proof.

\medskip\noindent\emph{Organization of the paper.}
Section~\ref{sec:normalization} introduces the Gaussian multipliers and
proves the vector-valued estimate.
Section~\ref{sec:weighted-chaos} gives the weighted criterion,
and Section~\ref{sec:certificate} applies it to exact parameters.
Section~\ref{sec:parameter-optimization} proves the dual bound and discusses
the information lost by the criterion.
Appendix~\ref{app:verification} gives the two interval verifications.
The verification code and exact inputs are included as ancillary files;
the accompanying repository is
{\useOriginalUrlSetting\url{https://github.com/shengtaoguo/complex-grothendieck-certificates}}.

\section{Gaussian Hermite multipliers}
\label{sec:normalization}

\subsection{The operator formulation}

Let \((\Omega,\mu)\) be a probability space and let
\(T:L_2(\mu)\to L_2(\mu)\) be bounded and complex linear.
Its \(L_\infty\)-to-\(L_1\) operator norm is
\(\norm T_{\infty\to1}=\sup_{\abs f,\abs g\le1}\abs{\langle Tf,g\rangle}\).
For vector-valued functions, \(T\) acts coordinatewise and the pairing
includes the sum over coordinates.
The classical operator form of Grothendieck's inequality
\cite[Theorem~2.5]{Pisier2012} states that
\begin{equation*}
\abs{\langle TF,G\rangle}
\le
\KG\norm T_{\infty\to1}\norm F_\infty\norm G_\infty
\end{equation*}
for every \(d\ge1\) and \(F,G\in L_\infty(\mu;\C^d)\).
In particular, a vector-valued test with \(\norm F_\infty\le1\) gives
\begin{equation}
\KG\ge\frac{\abs{\langle TF,F\rangle}}{\norm T_{\infty\to1}}
\qquad(T\ne0).
\label{eq:operator-ratio}
\end{equation}
For every real number \(c\) strictly below the quotient in
\eqref{eq:operator-ratio}, the finite-space reduction in the cited proof
gives a finite complex matrix with Grothendieck ratio greater than \(c\).

\subsection{Complex Gaussian and Hermite decomposition}

We use the circular complex Gaussian probability measure
\(\dd\gamma_n(z)=\pi^{-n}e^{-\norm z_2^2}\dd z\),
where \(\dd z\) is \(2n\)-dimensional Lebesgue measure.
For \(p,q\in\mathbb N_0\), the normalized one-variable complex Hermite
polynomials are
\begin{equation*}
h_{p,q}(z)=
\frac1{\sqrt{p!q!}}
\sum_{j=0}^{\min(p,q)}
(-1)^j j!\binom pj\binom qj z^{p-j}\overline z^{q-j}.
\end{equation*}
For example, \(h_{1,0}(z)=z\), \(h_{0,1}(z)=\overline z\), and
\(h_{1,1}(z)=\abs z^2-1\).
Their products form an orthonormal basis of \(L_2(\gamma_n)\)
\cite[Theorem~15]{Ito1952}.
For multi-indices \(\alpha,\beta\in\mathbb N_0^n\), put
\begin{equation*}
h_{\alpha,\beta}(z)=\prod_{j=1}^n h_{\alpha_j,\beta_j}(z_j),
\qquad
\mathcal H_{p,q}^{(n)}
=\operatorname{span}\{h_{\alpha,\beta}:|\alpha|=p,\ |\beta|=q\},
\end{equation*}
and let \(P_{p,q}^{(n)}\) be the orthogonal projection onto this subspace.
The indices \(p,q\) record its holomorphic and antiholomorphic Hermite degrees.
For the coordinate decompositions below, take \(\gamma_0\) to be the point
mass on \(\C^0\), with
\(\mathcal H_{0,0}^{(0)}=\C\) and all other
\(\mathcal H_{p,q}^{(0)}=\{0\}\).

Unitary changes of variables preserve the Gaussian measure and the spaces
of polynomials of bidegree at most \((p,q)\). They therefore preserve the
orthogonal Hermite components \(\mathcal H_{p,q}^{(n)}\), so each
\(P_{p,q}^{(n)}\) commutes with the action of \(U(n)\).

Scalar rotations satisfy
\(h_{\alpha,\beta}(e^{i\theta}z)
=e^{i(|\alpha|-|\beta|)\theta}h_{\alpha,\beta}(z)\).
We use the projections \(Q_\ell^{(n)}=P_{\ell+1,\ell}^{(n)}\) and
\(\Pi_1^{(n)}=\sum_{\ell\ge0}Q_\ell^{(n)}\),
where the sum converges strongly.
The range of \(\Pi_1^{(n)}\) consists of functions transforming as
\(f(e^{i\theta}z)=e^{i\theta}f(z)\).
Equivalently,
\[
\Pi_1^{(n)}f(z)
=\frac1{2\pi}\int_0^{2\pi}e^{-i\theta}f(e^{i\theta}z)\dd\theta.
\]
This formula gives contractivity on \(L_1\) and \(L_\infty\).
Consequently, compressing a bounded operator to this range cannot increase
its \(L_\infty\)-to-\(L_1\) norm, and preserves its pairing with
vector-valued tests whose coordinates lie in the range.
We suppress the superscript \((n)\) when the dimension is clear.

\subsection{The vector-valued test}

Fix a finite set \(\Lambda\subset\mathbb N_{\ge1}\), possibly empty,
\(0<\rho<1\), and coefficients \(\beta_\ell>0\) for \(\ell\in\Lambda\).
Throughout the paper,
\begin{equation}
T_n=Q_0-\rho\Pi_1-\sum_{\ell\in\Lambda}\beta_\ell Q_\ell.
\label{eq:multiplier}
\end{equation}
These parameters are independent of \(n\).

\begin{proposition}\label{prop:gaussian-vector-test}
Let \(F_n(z)=z/\norm z_2\) for \(z\ne0\), with \(F_n(0)=0\).
Put \(\beta_*=\max(\{0\}\cup\{\beta_\ell:\ell\in\Lambda\})\) and
\(C_T=\max\{1-\rho,\rho+\beta_*\}\).
Then
\begin{equation*}
\abs{\langle T_nF_n,F_n\rangle-(1-\rho)}
\le\frac{2C_T}{\sqrt n}.
\end{equation*}
In particular, \(\langle T_nF_n,F_n\rangle\to1-\rho\).
\end{proposition}

\begin{proof}
Put \(G_n(z)=z/\sqrt n\).
Both \(F_n\) and \(G_n\) have \(L_2(\gamma_n;\C^n)\) norm one.
For \(Z\sim\gamma_n\),
\begin{equation*}
\norm{F_n-G_n}_2^2
=
\E\left(1-\frac{\norm Z_2}{\sqrt n}\right)^2
\le
\E\left(1-\frac{\norm Z_2^2}{n}\right)^2
=\frac1n.
\end{equation*}
Here \(|1-x|\le|1-x^2|\) for \(x\ge0\), and the variables
\(|Z_j|^2\) are independent exponentials of variance one.
Since \(G_n\) lies in the range of \(Q_0\), we have
\(T_nG_n=(1-\rho)G_n\). The orthogonal Hermite decomposition gives
\(\norm{T_n}_{2\to2}=C_T\).
Consequently,
\begin{equation*}
\begin{aligned}
\abs{\langle T_nF_n,F_n\rangle-\langle T_nG_n,G_n\rangle}
&\le C_T\norm{F_n-G_n}_2(\norm{F_n}_2+\norm{G_n}_2)\\
&\le \frac{2C_T}{\sqrt n}.
\end{aligned}
\end{equation*}
The pairing with \(G_n\) is \(1-\rho\), which proves the assertion.
\end{proof}

\section{A weighted \texorpdfstring{\(L_\infty\)-to-\(L_1\)}{L-infinity-to-L1} estimate}
\label{sec:weighted-chaos}

For a fixed finite active set and a given weight, we give finitely many
sufficient conditions for an upper bound on the norm of
\eqref{eq:multiplier}.
Weighted Gram inequalities control the higher Hermite contributions;
one scalar Schur-complement condition characterizes a quadratic-form
inequality for radial functions.

\subsection{The criterion}

Let \(\dd\nu(r)=2re^{-r^2}\dd r\), \(r\ge0\),
be the radial law of one complex Gaussian coordinate.
Write \(h_{a,b}(re^{i\theta})=e^{i(a-b)\theta}e_{a,b}(r)\) and
\(\phi_\ell(r)=e_{\ell+1,\ell}(r)\).
The radial factors are real, and the functions \(\phi_\ell\) are orthonormal
in \(L_2(\nu)\), with \(\phi_0(r)=r\).

Let \(W:[0,\infty)\to(0,\infty)\) be measurable, with finite mean
\(\mu_W=\int_0^\infty W(r^2)\dd\nu(r)\).
For each \(p,q\in\mathbb N_0\), set
\(\mathcal I_{p,q}^{\Lambda}
=\{\ell\in\Lambda:p\le\ell+1,\ q\le\ell\}\).
These indices group the Hermite levels that contribute to the same
bidegree in the remaining coordinates.
Their first-coordinate angular frequency is \(1-p+q\), independent of
the active level. The Gram matrix below therefore retains the
correlations between their radial factors.
For every nonempty index set, assume the integrals below are finite and
define
\begin{equation}
\mathcal M_{\ell m}^{p,q}(W)
=
\sqrt{\beta_\ell\beta_m}
\int_0^\infty
\frac{e_{\ell+1-p,\ell-q}(r)e_{m+1-p,m-q}(r)}
     {W(r^2)}\dd\nu(r),
\qquad \ell,m\in\mathcal I_{p,q}^{\Lambda}.
\label{eq:residual-matrix}
\end{equation}
Only finitely many such matrices are nonempty.
Here and below, \(A\preceq B\) means that \(B-A\) is positive semidefinite.

For the radial quadratic form, set
\(D(r)=2\rho+W(r^2)\) and
\(B=\operatorname{diag}(\beta_\ell)_{\ell\in\Lambda}\),
and define
\begin{equation}
\begin{aligned}
R&=\int_0^\infty\frac{r^2}{D(r)}\dd\nu(r),
&
b_\ell&=\int_0^\infty\frac{r\phi_\ell(r)}{D(r)}\dd\nu(r),\\
G_{\ell m}&=\int_0^\infty
\frac{\phi_\ell(r)\phi_m(r)}{D(r)}\dd\nu(r).
\end{aligned}
\label{eq:woodbury-data}
\end{equation}
These integrals are finite since \(D\ge2\rho>0\).
With \(\boldsymbol b=(b_\ell)_{\ell\in\Lambda}\) and
\(G=(G_{\ell m})_{\ell,m\in\Lambda}\), put
\begin{equation}
S=R-\boldsymbol b^T(B^{-1}+G)^{-1}\boldsymbol b.
\label{eq:schur-complement}
\end{equation}
The inverse exists because \(B\succ0\) and \(G\succeq0\).
For \(\Lambda=\varnothing\), set \(S=R\).

\begin{theorem}[Weighted criterion]\label{thm:weighted-chaos}
Let \(\rho,\Lambda,(\beta_\ell)\) be as in \eqref{eq:multiplier}, and let
\(W\) satisfy the positivity and finiteness assumptions above.
Suppose that
\begin{equation}
\mathcal M^{p,q}(W)\preceq I
\quad\text{for every nonempty }\mathcal I_{p,q}^{\Lambda},
\label{eq:residual-condition}
\end{equation}
and \(S\le1\).
Then, for every \(n\ge1\),
\begin{equation}
\norm{T_n}_{\infty\to1}\le\rho+\mu_W.
\label{eq:scalar-bound}
\end{equation}
Consequently,
\begin{equation}
\KG\ge\frac{1-\rho}{\rho+\mu_W}.
\label{eq:weighted-lower-bound}
\end{equation}
\end{theorem}

The proof combines the next two estimates by polarization.

\subsection{The weighted Hermite estimate}
\label{sec:radial-formulas}

\begin{proposition}\label{prop:w-term-bound}
Suppose \(W>0\) is measurable, the matrices
\(\mathcal M^{p,q}(W)\) have finite entries, and
\eqref{eq:residual-condition} holds.
Then, for every \(n\ge1\) and \(w\in L_2(\gamma_n)\),
\begin{equation*}
\sum_{\ell\in\Lambda}\beta_\ell\norm{Q_\ell w}_2^2
\le
\int_{\C^n}W(|z_1|^2)|w(z)|^2\dd\gamma_n(z).
\end{equation*}
\end{proposition}

\begin{proof}
The claim is immediate if the right-hand side is infinite.
For a fixed \((p,q)\), the matrix \(\mathcal M^{p,q}(W)\) is the Gram
matrix in \(L_2(\nu)\) of
\[
\sqrt{\beta_\ell}\,
\frac{e_{\ell+1-p,\ell-q}}{\sqrt{W(r^2)}},
\qquad \ell\in\mathcal I_{p,q}^{\Lambda}.
\]
Thus \(\mathcal M^{p,q}(W)\preceq I\) is equivalent to the weighted
Bessel inequality
\begin{equation}
\sum_{\ell\in\mathcal I_{p,q}^{\Lambda}}\beta_\ell
\abs*{\int_0^\infty a(r)e_{\ell+1-p,\ell-q}(r)\dd\nu(r)}^2
\le
\int_0^\infty W(r^2)|a(r)|^2\dd\nu(r)
\label{eq:residual-bessel}
\end{equation}
for every \(a\in L_2(W\dd\nu)\), where the weight in this space is
\(W(r^2)\).

To apply this inequality, write \(z=(z_1,z')\) and use the tensor
decomposition
\begin{equation*}
\mathcal H_{\ell+1,\ell}^{(n)}
=
\bigoplus_{\substack{0\le p\le\ell+1\\0\le q\le\ell}}
\mathcal H_{\ell+1-p,\ell-q}^{(1)}
\otimes\mathcal H_{p,q}^{(n-1)}.
\end{equation*}
Choose an orthonormal basis
\((H_\sigma^{p,q})_{\sigma\in\Sigma_{p,q}}\) in each Hermite space
\(\mathcal H_{p,q}^{(n-1)}\) of the remaining coordinates, and define
\[
a_{p,q,\sigma}(r)
=
\E_{\theta,z'}\!\left[
w(re^{i\theta},z')e^{-i(1-p+q)\theta}
\overline{H_\sigma^{p,q}(z')}
\right],
\]
where \(\theta\) is uniform on \([0,2\pi)\) and \(z'\sim\gamma_{n-1}\).
The first-coordinate angular frequency is \(1-p+q\), independent of
\(\ell\).  Consequently, the same profile \(a_{p,q,\sigma}\) gives the
coefficients for every active level \(\ell\in\mathcal I_{p,q}^{\Lambda}\):
\[
\norm{Q_\ell w}_2^2
=
\sum_{\substack{0\le p\le\ell+1\\0\le q\le\ell}}
\sum_{\sigma\in\Sigma_{p,q}}
\abs*{\int_0^\infty
a_{p,q,\sigma}(r)e_{\ell+1-p,\ell-q}(r)\dd\nu(r)}^2.
\]
For almost every \(r\), the functions
\(e^{i(1-p+q)\theta}H_\sigma^{p,q}(z')\) form an orthonormal family.
Bessel's inequality therefore gives
\(\sum_{p,q,\sigma}|a_{p,q,\sigma}(r)|^2
\le \E_{\theta,z'}|w(re^{i\theta},z')|^2\).
When the right-hand side of the proposition is finite, every such profile
belongs to \(L_2(W\dd\nu)\).
Apply \eqref{eq:residual-bessel} to each profile and sum:
\[
\begin{aligned}
\sum_{\ell\in\Lambda}\beta_\ell\norm{Q_\ell w}_2^2
&\le
\int_0^\infty W(r^2)\sum_{p,q,\sigma}|a_{p,q,\sigma}(r)|^2\dd\nu(r)\\
&\le
\int_{\C^n}W(|z_1|^2)|w(z)|^2\dd\gamma_n(z).
\end{aligned}
\]
The exchanges of nonnegative sums and integrals are justified by Tonelli's
theorem.
\end{proof}

\subsection{The radial quadratic form}

The scalar \(S\) in \eqref{eq:schur-complement} is the squared norm of
the linear functional \(t\mapsto\int_0^\infty rt(r)\dd\nu(r)\) for the
weighted quadratic form below.

\begin{lemma}\label{lem:rank-one-schur}
With \(D,B,R,\boldsymbol b,G,S\) as above,
\begin{equation*}
S=
\sup_{0\ne t\in L_2(D\dd\nu)}
\frac{\displaystyle\abs*{\int_0^\infty rt\dd\nu}^2}
{\displaystyle
 \int_0^\infty D|t|^2\dd\nu+
 \sum_{\ell\in\Lambda}\beta_\ell
 \abs*{\int_0^\infty\phi_\ell t\dd\nu}^2}.
\end{equation*}
In particular, \(S\le1\) is equivalent to
\begin{equation}
\abs*{\int_0^\infty rt\dd\nu}^2
\le
\int_0^\infty D|t|^2\dd\nu
+\sum_{\ell\in\Lambda}\beta_\ell
\abs*{\int_0^\infty\phi_\ell t\dd\nu}^2
\qquad\bigl(t\in L_2(D\dd\nu)\bigr).
\label{eq:one-dimensional-form}
\end{equation}
\end{lemma}

\begin{proof}
The denominator is the quadratic form of the inner product
\[
\mathfrak q(t,h)
=
\int_0^\infty Dt\overline h\dd\nu
+\sum_{\ell\in\Lambda}\beta_\ell
\left(\int_0^\infty\phi_\ell t\dd\nu\right)
\overline{\left(\int_0^\infty\phi_\ell h\dd\nu\right)}
\]
on \(L_2(D\dd\nu)\).
All moment functionals are continuous because \(D\ge2\rho>0\).
For \(\Lambda\ne\varnothing\), set
\[
\boldsymbol c=(B^{-1}+G)^{-1}\boldsymbol b,
\qquad
t_*(r)=\frac{r-\sum_{\ell\in\Lambda}c_\ell\phi_\ell(r)}{D(r)}.
\]
Then \(t_*\in L_2(D\dd\nu)\), and
all coefficients \(c_\ell\) are real.  Moreover,
\(\int_0^\infty\phi_\ell t_*\dd\nu
=b_\ell-(G\boldsymbol c)_\ell=c_\ell/\beta_\ell\).
Substitution gives \(\mathfrak q(t,t_*)=\int_0^\infty rt\dd\nu\) and
\(\mathfrak q(t_*,t_*)=R-\boldsymbol b^T\boldsymbol c=S\).
Cauchy--Schwarz now bounds the quotient by \(S\), with equality at \(t=t_*\).
This function is nonzero because \(\phi_0=r\) is orthogonal to all active
\(\phi_\ell\).
For \(\Lambda=\varnothing\), take \(t_*=r/D\); the same calculation gives
\(S=R\).
\end{proof}

\subsection{Proof of the weighted criterion}

\begin{proof}[Proof of Theorem~\ref{thm:weighted-chaos}]
Every function \(h\) in the unit ball of \(L_\infty(\gamma_n)\) is the
average of the two unimodular functions
\(\xi(|h|\pm i\sqrt{1-|h|^2})\), where \(\xi=h/|h|\) for \(h\ne0\)
and \(\xi=1\) otherwise. Applying this representation to each input gives
\(\norm{T_n}_{\infty\to1}
=\sup_{|f|=|g|=1\ \mathrm{a.e.}}|\langle T_nf,g\rangle|\).
Fix such \(f,g\). Multiplying \(g\) by a constant of modulus one, we may
assume that \(\langle T_nf,g\rangle\) is real and nonnegative.
Put \(u=(f+g)/2\) and \(w=(f-g)/2\).
They satisfy \(|u|^2+|w|^2=1\) and \(\RePart(u\overline w)=0\).
Self-adjointness gives
\begin{equation}
\begin{aligned}
\RePart\langle T_nf,g\rangle
={}&
\norm{Q_0u}_2^2-\norm{Q_0w}_2^2
-\rho\norm{\Pi_1u}_2^2+\rho\norm{\Pi_1w}_2^2\\
&-\sum_{\ell\in\Lambda}\beta_\ell\norm{Q_\ell u}_2^2
+\sum_{\ell\in\Lambda}\beta_\ell\norm{Q_\ell w}_2^2.
\end{aligned}
\label{eq:exact-symmetrized-pairing}
\end{equation}
Since
\(\norm{\Pi_1w}_2^2\le1-\norm u_2^2\le1-\norm{\Pi_1u}_2^2\),
discarding \(-\norm{Q_0w}_2^2\) yields
\begin{equation}
\begin{aligned}
\RePart\langle T_nf,g\rangle
\le{}&\rho+\norm{Q_0u}_2^2-2\rho\norm{\Pi_1u}_2^2\\
&-\sum_{\ell\in\Lambda}\beta_\ell\norm{Q_\ell u}_2^2
+\sum_{\ell\in\Lambda}\beta_\ell\norm{Q_\ell w}_2^2.
\end{aligned}
\label{eq:parallelogram-bound}
\end{equation} 

The multiplier and Gaussian measure are invariant under \(U(n)\).
We may therefore rotate the variables so that \(Q_0u(z)=Az_1\), where
\(A=\norm{Q_0u}_2\ge0\).
Apply Proposition~\ref{prop:w-term-bound} in this coordinate.
Using \(|w|^2=1-|u|^2\), we obtain
\[
\sum_{\ell\in\Lambda}\beta_\ell\norm{Q_\ell w}_2^2
\le
\mu_W-\int_{\C^n}W(|z_1|^2)|u(z)|^2\dd\gamma_n(z).
\]
Write \(z=(re^{i\theta},z')\).
For uniform \(\theta\in[0,2\pi)\) and
\(z'\sim\gamma_{n-1}\), set
\(t(r)=\E_{\theta,z'}[e^{-i\theta}u(re^{i\theta},z')]\).
The function \(e^{i\theta}t(r)\) is the orthogonal projection of \(u\) onto
the functions depending only on \(z_1\) and transforming with angular
frequency one.
This subspace is contained in the range of \(\Pi_1\), which gives the
second line of \eqref{eq:t-moments} below.
Also, \(h_{\ell+1,\ell}(z_1)=e^{i\theta}\phi_\ell(r)\) is a unit vector in
the range of \(Q_\ell\).
It follows that
\begin{equation}
\begin{aligned}
\norm{Q_0u}_2^2&=\abs*{\int_0^\infty rt(r)\dd\nu(r)}^2,\\
\norm{\Pi_1u}_2^2&\ge\int_0^\infty|t(r)|^2\dd\nu(r),\\
\norm{Q_\ell u}_2^2&\ge
\abs*{\int_0^\infty\phi_\ell(r)t(r)\dd\nu(r)}^2.
\end{aligned}
\label{eq:t-moments}
\end{equation}
Jensen's inequality, applied to the average defining \(t\), gives
\begin{equation}
\int_{\C^n}W(|z_1|^2)|u(z)|^2\dd\gamma_n(z)
\ge
\int_0^\infty W(r^2)|t(r)|^2\dd\nu(r).
\label{eq:w-jensen}
\end{equation}
Combining these inequalities, with \(D(r)=2\rho+W(r^2)\), gives
\[
\RePart\langle T_nf,g\rangle-(\rho+\mu_W)
\le
\abs*{\int_0^\infty rt\dd\nu}^2
-\int_0^\infty D|t|^2\dd\nu
-\sum_{\ell\in\Lambda}\beta_\ell
\abs*{\int_0^\infty\phi_\ell t\dd\nu}^2
\le0.
\]
Indeed, \(|t|\le1\) and \(\int D\dd\nu=2\rho+\mu_W<\infty\), so
Lemma~\ref{lem:rank-one-schur} applies.
Our choice of phase therefore proves \eqref{eq:scalar-bound}.
Finally, apply \eqref{eq:operator-ratio} to \(F_n\) and use
Proposition~\ref{prop:gaussian-vector-test} as \(n\to\infty\) to obtain
\eqref{eq:weighted-lower-bound}.
\end{proof}

\section{An explicit lower bound}
\label{sec:certificate}

We now give exact parameters satisfying the Gram and Schur conditions of
Theorem~\ref{thm:weighted-chaos}.  Their ratio
\((1-\rho)/(\rho+\mu_W)\) proves Theorem~\ref{thm:main-lower-bound}.

\subsection{Exact parameters}
\label{sec:exact-parameters}

Take \(\Lambda=\Lambda_*:=\{1,\ldots,10\}\),
and interpret every decimal in this section as the corresponding exact
rational number.  The multiplier is determined by
\(\Lambda_*,\rho,(\beta_\ell)_{\ell\in\Lambda_*}\), while \(W\) is the
auxiliary weight used to bound its \(L_\infty\)-to-\(L_1\) norm.

Set \(\rho=0.012039377730497942\).
The coefficients are
\begin{equation*}
\begin{array}{c@{\qquad}l@{\qquad\qquad}c@{\qquad}l}
\ell&\beta_\ell&\ell&\beta_\ell\\ \hline
1&0.4283744445791419&6&0.026165030774068086\\
2&0.1506077793447579&7&0.01945831930582338\\
3&0.08328691492582467&8&0.01490133597851163\\
4&0.05294104124824108&9&0.011682554694871422\\
5&0.03632381026351961&10&0.014696581084273385
\end{array}
\end{equation*}
These parameters define, for every \(n\ge1\), the multiplier
\(T_n=Q_0-\rho\Pi_1-\sum_{\ell=1}^{10}\beta_\ell Q_\ell\).
We use the positive rational function
\begin{equation}
W(s)
=
\alpha+\eta s
+\zeta_1\frac{s}{\delta_1+s}
+\zeta_2\frac{s}{\delta_2+s},
\label{eq:weight-candidate}
\end{equation}
where
\begin{equation*}
\begin{aligned}
\alpha&=0.16212090047056932029265504,
&
\eta&=0.2995274614735440217489782,\\
\zeta_1&=0.1705006229073915451327806,
&
\delta_1&=0.6122395153199937,\\
\zeta_2&=1.3819822296622917537966546,
&
\delta_2&=6.340370329139007.
\end{aligned}
\end{equation*} 

The displayed parameters were obtained by numerical search, as outlined
in Appendix~\ref{app:candidate-discovery}. The verification
below uses the exact rational data above.

\subsection{Verification and consequences}

All coefficients in \eqref{eq:weight-candidate} are positive, so
\(W(s)\ge\alpha>0\) for \(s\ge0\), while
\(W(s)=\eta s+O(1)\) as \(s\to\infty\).  Since the relevant Hermite factors
have polynomial growth and \(\dd\nu=e^{-s}\dd s\) after the substitution
\(s=r^2\), every entry of \(\mathcal M^{p,q}(W)\), as well as
\(\mu_W\), \(R\), \(\boldsymbol b\), and \(G\), is finite.

\begin{proposition}
\label{prop:directed-certificate}
For the exact data in Section~\ref{sec:exact-parameters},
\(I-\mathcal M^{p,q}(W)\succ0\) for every nonempty block
\(\mathcal I_{p,q}^{\Lambda_*}\), and \(S<1\).
The weight mean and ratio satisfy
\begin{align}
\mu_W&<0.716627735168304,
\label{eq:mu-upper}\\
\frac{1-\rho}{\rho+\mu_W}
&>1.35584631827168.
\label{eq:certified-ratio}
\end{align}
\end{proposition}

\begin{proof}[Computer-assisted proof]
Appendix~\ref{app:verification} gives the code, exact inputs, and details
of the interval calculation.
After the substitution \(s=r^2\), the required integrals reduce to finite
linear combinations of exponential-integral moments, evaluated with outward
rounding from the exact rational inputs.

For each matrix \(I-\mathcal M^{p,q}(W)\), interval arithmetic encloses the
pivots of its exact \(LDL^*\) decomposition in a fixed elimination order.
Let \(\underline d_{p,q,j}\) denote the resulting lower endpoint for its
\(j\)-th pivot.  The calculation gives
\begin{equation}
\min_{p,q,j}\underline d_{p,q,j}
>
1.2259567912\times10^{-5}.
\label{eq:residual-pivot-bound}
\end{equation}
Positivity of these exact pivots proves
\(I-\mathcal M^{p,q}(W)\succ0\) by the \(LDL^*\) criterion.
The bound in \eqref{eq:residual-pivot-bound} is a pivot bound for the fixed
elimination order, not an eigenvalue bound.

The interval enclosure of \(S\) gives
\(1-S>1.8374949151\times10^{-9}\),
and hence \(S<1\).  The same outward-rounded calculation encloses \(\mu_W\)
and \((1-\rho)/(\rho+\mu_W)\), yielding \eqref{eq:mu-upper} and
\eqref{eq:certified-ratio}.
\end{proof}

\begin{proof}[Proof of Theorem~\ref{thm:main-lower-bound}]
Proposition~\ref{prop:directed-certificate} and
Theorem~\ref{thm:weighted-chaos} give
\(\KG>1.35584631827168\).
\end{proof}

\section{A dual bound for the weighted criterion}
\label{sec:parameter-optimization}

We now ask whether additional Hermite levels or a more flexible radial weight
can substantially improve the preceding bound.  For the sufficient conditions
of Theorem~\ref{thm:weighted-chaos}, a dual pair gives a uniform upper bound
for \((1-\rho)/(\rho+\mu_W)\) over all admissible parameter choices.

\subsection{The restricted optimum}

\begin{definition}[Restricted optimization value]\label{def:restricted-optimum}
Let \(\Kstar\) be the supremum of \((1-\rho)/(\rho+\mu_W)\)
over every finite set \(\Lambda\subset\mathbb N_{\ge1}\), including the
empty set, every \(0<\rho<1\), every family
\((\beta_\ell)_{\ell\in\Lambda}\) of positive coefficients, and every
measurable positive radial
weight \(W\) with \(\mu_W<\infty\) and finite entries in the Gram matrices,
such that
\(\mathcal M^{p,q}(W)\preceq I_{\lvert\mathcal I_{p,q}^{\Lambda}\rvert}\)
for every nonempty block \(\mathcal I_{p,q}^{\Lambda}\), and \(S\le 1\).
\end{definition}

By Theorem~\ref{thm:weighted-chaos}, \(\Kstar\le\KG\).
The optimization includes every finite active set and every admissible
positive weight, rather than only the choices used in
Section~\ref{sec:certificate}.

\begin{theorem}[Dual bound]
\label{thm:restricted-dual-bound}
The restricted optimization value satisfies
\[
1.35584631827168
<
\Kstar
<
1.35584697425050.
\]
\end{theorem}

The lower inequality follows from Section~\ref{sec:certificate}.
The upper inequality leaves less than \(6.56\times10^{-7}\) for improvement
within this optimization problem.

\subsection{The dual argument}

We combine \(S\le1\) with \(\mathcal M^{2,1}(W)\preceq I\).
The block \((2,1)\) contains every active index, since
\(\mathcal I_{2,1}^{\Lambda}=\Lambda\).
We choose tests with matching coefficients in the radial bases of these two
inequalities.  This will eliminate the terms containing \(\beta_\ell\),
leaving a bound involving only \(\rho\) and \(\mu_W\).

Set \(s=r^2\) and \(\dd\widehat\nu(s)=e^{-s}\dd s\), and write
\(\mathcal H_{\mathrm{rad}}=L_2([0,\infty),\widehat\nu)\).
Let \(L_j^{(\alpha)}\) denote the generalized Laguerre polynomials.
The condition \(S\le1\) and the block inequality
\(\mathcal M^{2,1}(W)\preceq I\) use, respectively, the functions
\begin{align*}
\widehat\phi_\ell(s)
&=\phi_\ell(\sqrt s)
=\frac{(-1)^\ell\sqrt s}{\sqrt{\ell+1}}L_\ell^{(1)}(s),
&&\ell\ge0,\\
\chi_\ell(s)&=e_{\ell-1,\ell-1}(\sqrt s)
=(-1)^{\ell-1}L_{\ell-1}^{(0)}(s),
&&\ell\ge1.
\end{align*}
Both families are orthonormal bases of \(\mathcal H_{\mathrm{rad}}\): for
\((\chi_\ell)_{\ell\ge1}\), by standard Laguerre orthogonality; for
\((\widehat\phi_\ell)_{\ell\ge0}\), because multiplication by \(\sqrt{s}\)
carries the orthonormal basis
\(\bigl((-1)^\ell L_\ell^{(1)}/\sqrt{\ell+1}\bigr)_{\ell\ge0}\) of
\(L_2([0,\infty),s e^{-s}\dd s)\) onto the displayed family; see
\cite[Sections~5.1 and~5.7]{Szego1975}.

\begin{lemma}[Dual upper bound]\label{lem:dual-upper-bound}
Suppose real functions \(X,Y\in\mathcal H_{\mathrm{rad}}\) satisfy
\begin{equation}
\langle\widehat\phi_0,Y\rangle=1,
\qquad
\langle\chi_\ell,X\rangle
=
\langle\widehat\phi_\ell,Y\rangle
\quad(\ell\ge1),
\label{eq:coefficient-matching}
\end{equation}
and \(X(s)^2+Y(s)^2\le C\) for almost every \(s\ge0\).
Then \(\Kstar\le C\).
\end{lemma}

\begin{proof}
Fix any admissible parameter choice from
Definition~\ref{def:restricted-optimum}.
The weighted Bessel inequality \eqref{eq:residual-bessel} for \((2,1)\) gives
\begin{equation*}
\sum_{\ell\in\Lambda}\beta_\ell
\abs{\langle f,\chi_\ell\rangle}^2
\le
\int_0^\infty W(s)\abs{f(s)}^2\dd\widehat\nu(s),
\end{equation*}
while \(S\le1\), in the form \eqref{eq:one-dimensional-form}, gives
\begin{equation*}
\abs{\langle h,\widehat\phi_0\rangle}^2
\le
2\rho\norm h_2^2
+\int_0^\infty W(s)\abs{h(s)}^2\dd\widehat\nu(s)
+\sum_{\ell\in\Lambda}\beta_\ell
\abs{\langle h,\widehat\phi_\ell\rangle}^2.
\end{equation*}
Both hold for functions in \(\mathcal H_{\mathrm{rad}}\) with finite
weighted squared norm.  They apply to \(f=X\) and \(h=Y\), since
\(\int W(X^2+Y^2)\dd\widehat\nu
\le C\int W\dd\widehat\nu=C\mu_W<\infty\).
The matching coefficients make the sums in these two inequalities equal.
Substitution therefore eliminates every \(\beta_\ell\):
\begin{equation}
\begin{aligned}
1
&\le
2\rho\norm{Y}_2^2
+\int W(X^2+Y^2)\dd\widehat\nu\\
&\le
2\rho\norm{Y}_2^2+C\mu_W.
\end{aligned}
\label{eq:dual-first-estimate}
\end{equation}
Parseval and \eqref{eq:coefficient-matching} yield
\(\norm{X}_2^2=\norm{Y}_2^2-1\).
Integrating the pointwise bound against the probability measure
\(\widehat\nu\) therefore gives
\(2\norm{Y}_2^2-1=\norm{X}_2^2+\norm{Y}_2^2\le C\).
Substitution in \eqref{eq:dual-first-estimate} yields
\(1\le(C+1)\rho+C\mu_W\), or equivalently
\((1-\rho)/(\rho+\mu_W)\le C\).
Taking the supremum over all admissible parameter choices proves the claim.
\end{proof}

\subsection{Dual pairs from finitely many parameters}

We enforce coefficient matching by using a common sequence in the two
Laguerre expansions. We choose an alternating coefficient sequence whose
magnitudes are positive mixtures of geometric sequences; this gives
explicit formulas for the resulting functions.

\begin{lemma}[Geometric mixtures]
\label{lem:atomic-matching-pair}
Let \(N\ge1\) be an integer, and choose numbers \(q_1,\ldots,q_N\) and
\(c_1,\ldots,c_N\) satisfying \(c_j>0\) and \(0<q_j<1\)
for \(1\le j\le N\), and \(\sum_{j=1}^N c_j=1\).
Define \(a_\ell=(-1)^\ell\sum_{j=1}^N c_jq_j^\ell\) for \(\ell\ge0\)
and
\begin{equation}
Y=\sum_{\ell=0}^\infty a_\ell\widehat\phi_\ell,
\qquad
X=\sum_{\ell=1}^\infty a_\ell\chi_\ell.
\label{eq:dual-expansions}
\end{equation}
\begin{samepage}
Then \(X,Y\in\mathcal H_{\mathrm{rad}}\) satisfy
\eqref{eq:coefficient-matching}. They admit the pointwise representations
\begin{align}
X(s)
&=
-\sum_{j=1}^N
c_j\frac{q_j}{1-q_j}
\exp\left(-s\frac{q_j}{1-q_j}\right),
\label{eq:X-formula}\\
Y(s)
&=
\frac{2\sqrt s}{\sqrt\pi}
\sum_{j=1}^N c_j
\int_0^\infty
\frac{e^{-v^2}}{(1-q_je^{-v^2})^2}
\times
\exp\left(
-s\frac{q_je^{-v^2}}{1-q_je^{-v^2}}
\right)
\dd v.
\label{eq:Y-formula}
\end{align}
\end{samepage}
\end{lemma}

\begin{proof}
Let \(q_* = \max_j q_j<1\).  Since
\(\abs{a_\ell}\le q_*^\ell\), both coefficient sequences in
\eqref{eq:dual-expansions} are square summable.  The orthonormality of the
two Laguerre families therefore proves that these series define
\(X,Y\in\mathcal H_{\mathrm{rad}}\).  It also gives
\(\langle\widehat\phi_0,Y\rangle=a_0=1\) and, for every \(\ell\ge1\),
\(\langle\chi_\ell,X\rangle
=\langle\widehat\phi_\ell,Y\rangle=a_\ell\).

Equation \eqref{eq:X-formula} follows from the Laguerre generating function.
For \eqref{eq:Y-formula}, insert
\(1/\sqrt{\ell+1}=(2/\sqrt\pi)\int_0^\infty e^{-(\ell+1)v^2}\dd v\)
into \eqref{eq:dual-expansions} and use the generating function for
\(L_\ell^{(1)}\).  Since
\(q_je^{-v^2}\le q_*<1\), the resulting generating series are absolutely
and locally uniformly convergent in \(s\); dominated convergence therefore
justifies the sum and integral interchanges.
\end{proof}

\begin{samepage}
\begin{proposition}
\label{prop:bounded-dual-pair}
There exist \(180\) pairs of rational numbers \((c_j,q_j)\), represented
by terminating decimals, with
\(c_j>0\), \(0<q_j<1\) for \(1\le j\le180\), and
\(\sum_{j=1}^{180}c_j=1\),
such that the corresponding pair \((X,Y)\) from
Lemma~\ref{lem:atomic-matching-pair} satisfies
\eqref{eq:coefficient-matching} and
\begin{equation}
\sup_{s\ge0}\bigl(X(s)^2+Y(s)^2\bigr)
<1.35584697425050.
\label{eq:dual-cover-bound}
\end{equation}
\end{proposition}
\end{samepage}

\begin{proof}[Computer-assisted proof]
The \(180\) exact rational pairs \((c_j,q_j)\) specified in
Appendix~\ref{app:verification} satisfy
\(c_j>0\), \(0<q_j<1\), and \(\sum_jc_j=1\), as verified in exact arithmetic.
Lemma~\ref{lem:atomic-matching-pair} then gives the coefficient matching.
For the pointwise estimate \eqref{eq:dual-cover-bound}, write
\(Y(s)=\sqrt{s}H(s)\).  The quantity
\(X(s)^2+sH(s)^2\) is smooth at the origin and can be bounded by Taylor
estimates on \(0\le s\le10^5\).  An analytic decay estimate handles
\(s\ge10^5\).  Appendix~\ref{app:dual-verification} gives the interval
calculation and the remainder bounds.
\end{proof}

\begin{proof}[Proof of Theorem~\ref{thm:restricted-dual-bound}]
The bounds in Section~\ref{sec:certificate} give
\(\Kstar>1.35584631827168\) by \eqref{eq:certified-ratio}.
Proposition~\ref{prop:bounded-dual-pair} and
Lemma~\ref{lem:dual-upper-bound} give \(\Kstar<1.35584697425050\).
\end{proof}

\subsection{Where the criterion loses information}
\label{sec:beyond-certificate}

The upper bound for \(\Kstar\) concerns the weighted criterion, rather than
the original operator problem.
Potential losses occur in four estimates.

\begin{enumerate}
\item In passing from \eqref{eq:exact-symmetrized-pairing} to
\eqref{eq:parallelogram-bound}, we discard
\(-\norm{Q_0w}_2^2\) and use
\(\norm{\Pi_1w}_2^2\le\norm w_2^2
=1-\norm u_2^2\le1-\norm{\Pi_1u}_2^2\).
These inequalities need not be equalities.

\item The positive term
\(\sum_{\ell\in\Lambda}\beta_\ell\norm{Q_\ell w}_2^2\) is bounded by
\(\int W(|z_1|^2)|w|^2\dd\gamma_n\).
This estimate is imposed independently of \(u\), even though \(u,w\)
come from the same unimodular pair and satisfy
\(\RePart(u\overline w)=0\).

\item The projection estimates \eqref{eq:t-moments} use only one radial
component of \(u\). Other components contribute further nonpositive terms
to the upper bound.

\item Jensen's inequality in \eqref{eq:w-jensen} replaces the weighted
energy of \(u\) by that of its radial profile.  The difference is
\[
\int_0^\infty W(r^2)\,
\E_{\theta,z'}\!
\left|e^{-i\theta}u(re^{i\theta},z')-t(r)\right|^2\dd\nu(r).
\]
\end{enumerate}

Requiring the homogeneous radial inequality \eqref{eq:one-dimensional-form}
for every \(t\in L_2(D\dd\nu)\), rather than only for realizable profiles
with \(|t|\le1\), introduces no additional loss in that inequality.
Indeed, every measurable \(t\) with \(|t|\le1\) is realizable.
With \(\xi=t/|t|\) where \(t\ne0\), and \(\xi=1\) otherwise, take
\[
u(re^{i\theta},z')=e^{i\theta}t(r),
\qquad
w(re^{i\theta},z')=ie^{i\theta}\xi(r)\sqrt{1-|t(r)|^2}.
\]
Then \(f=u+w\) and \(g=u-w\) are unimodular and give the prescribed
profile.  A common phase makes its first moment nonnegative without
changing the quadratic inequality.
Homogeneity extends the test to all bounded profiles, and bounded
truncation extends it to \(L_2(D\dd\nu)\), since the moment functionals
are continuous there.

The Gram conditions also limit what can be gained by changing the
vector-valued test. Under the hypotheses of Theorem~\ref{thm:weighted-chaos},
the \((\ell,\ell)\) entry of the Gram condition for
\((p,q)=(\ell+1,\ell)\), together with \(e_{0,0}=1\) and
Cauchy--Schwarz, gives
\[
\beta_\ell
\le\left(\int_0^\infty\frac{\dd\nu(r)}{W(r^2)}\right)^{-1}
\le\mu_W
\qquad(\ell\in\Lambda).
\]
With \(\beta_*\) and \(C_T\) as in
Proposition~\ref{prop:gaussian-vector-test}, this gives
\(\beta_*\le\mu_W\), also when \(\Lambda=\varnothing\).
For any \(n,d\ge1\) and \(F,G\in L_\infty(\gamma_n;\C^d)\) with
\(\norm F_\infty,\norm G_\infty\le1\), we have
\(\norm F_2,\norm G_2\le1\). The spectral bound therefore yields
\[
\frac{\abs{\langle T_nF,G\rangle}}{\rho+\mu_W}
\le\frac{C_T}{\rho+\mu_W}
\le\max\left\{\frac{1-\rho}{\rho+\mu_W},1\right\}
\le\Kstar,
\]
where the last step uses Definition~\ref{def:restricted-optimum} and the lower bound
from Section~\ref{sec:certificate}.
As long as the norm bound \eqref{eq:scalar-bound} is retained, changing the
bounded vector-valued test cannot produce a lower bound larger than
\(\Kstar\).
An improvement beyond \(1.35584697425050\) within the multiplier family
\eqref{eq:multiplier} requires a sharper estimate of
\(\norm{T_n}_{\infty\to1}\).
Joint information about \(u,w\) may lead to a sharper estimate.
Unimodularity alone imposes no further restriction on \(t\), since every
measurable disk-valued radial profile is realizable.

\appendix

\section{Computer-assisted verification}
\label{app:verification}

This appendix gives the interval calculations used in
Propositions~\ref{prop:directed-certificate} and~\ref{prop:bounded-dual-pair}.
The first verifies the Gram inequalities, Schur condition, and lower-bound
ratio; the second verifies the pointwise bound for the matching pair used
to bound \(\Kstar\).

The code and exact inputs are included as ancillary files.
See the code and input files on GitHub at the following paths, relative
to the repository root:
\begin{center}
\begin{tabular}{@{}lll@{}}
Proposition~\ref{prop:directed-certificate} & Code &
\href{https://github.com/shengtaoguo/complex-grothendieck-certificates/blob/main/verification/computations/lower_bound.py}{\nolinkurl{verification/computations/lower_bound.py}} \\
& Input &
\href{https://github.com/shengtaoguo/complex-grothendieck-certificates/blob/main/verification/data/lower_bound.json}{\nolinkurl{verification/data/lower_bound.json}} \\[4pt]
Proposition~\ref{prop:bounded-dual-pair} & Code &
\href{https://github.com/shengtaoguo/complex-grothendieck-certificates/blob/main/verification/computations/dual_bound.py}{\nolinkurl{verification/computations/dual_bound.py}} \\
& Input &
\href{https://github.com/shengtaoguo/complex-grothendieck-certificates/blob/main/verification/data/dual_pair.json}{\nolinkurl{verification/data/dual_pair.json}}.
\end{tabular}
\end{center}
The entry point
\href{https://github.com/shengtaoguo/complex-grothendieck-certificates/blob/main/verification/verify.py}{\texttt{verification/verify.py}}
runs both computations and their checks; the accompanying README gives
the prerequisites and running instructions.

\subsection{The lower-bound computation}
\label{app:candidate-discovery}

The candidate was found numerically by restricting the search to
\[
\Lambda=\{1,\ldots,L\},
\qquad
W(s)=\alpha+\eta s+\sum_{j=1}^{P}\zeta_j\frac{s}{\delta_j+s},
\]
and maximizing \((1-\rho)/(\rho+\mu_W)\) by constrained numerical
optimization, with Gauss--Laguerre quadrature for the integrals.
The final choice is \(L=10\) and \(P=2\).
We slightly increased the amplitudes in \(W\), keeping \(\rho\),
the \(\beta_\ell\), and the pole locations fixed, before rounding to the
exact rational data in Section~\ref{sec:exact-parameters}.

All interval computations below use Arb ball arithmetic~\cite{Johansson2017},
with outward rounding so that each interval contains the exact quantity
being computed.
The resulting pivot and Schur margins, weight mean and lower-bound ratio
are stated in Proposition~\ref{prop:directed-certificate} and its proof.
We describe how they are computed below.

\subsubsection{Evaluation of the integrals}
\label{app:lb-integrals}

To verify the Gram and Schur conditions, we evaluate
two collections of integrals.  For
\(\ell,m\in\mathcal I_{p,q}^{\Lambda}\), the Gram-matrix entries are
\[
\mathcal M_{\ell m}^{p,q}(W)
=
\sqrt{\beta_\ell\beta_m}
\int_0^\infty
\frac{e_{\ell+1-p,\ell-q}(r)e_{m+1-p,m-q}(r)}{W(r^2)}
\dd\nu(r).
\]
For the Schur condition, put \(D(r)=2\rho+W(r^2)\) and compute
\[
\begin{aligned}
R&=\int_0^\infty\frac{r^2}{D(r)}\dd\nu(r),
& b_\ell&=\int_0^\infty\frac{r\phi_\ell(r)}{D(r)}\dd\nu(r),\\
G_{\ell m}&=\int_0^\infty
\frac{\phi_\ell(r)\phi_m(r)}{D(r)}\dd\nu(r).
\end{aligned}
\]
Every entry contains either \(1/W(r^2)\) or \(1/(2\rho+W(r^2))\).
Partial fractions reduce these integrals to the moments \(J_k(\lambda)\)
defined below.

Write \(s=r^2\), so \(\dd\nu(r)=e^{-s}\dd s\).  Clearing the two positive
denominators in \eqref{eq:weight-candidate}, put
\(\Delta(s)=(s+\delta_1)(s+\delta_2)\) and \(N(s)=\Delta(s)W(s)\).
Both \(N\) and \(N+2\rho\Delta\) are polynomials of degree three, and
\[
\frac1{W(s)}=\frac{\Delta(s)}{N(s)},
\qquad
\frac1{2\rho+W(s)}
=
\frac{\Delta(s)}{N(s)+2\rho\Delta(s)}.
\]
Both denominator polynomials have real, simple, negative roots.
Indeed, \(0<\delta_1<\delta_2\)
and the rational expression for \(W\) satisfies
\[
W'(s)=\eta+\sum_{j=1}^2\frac{\zeta_j\delta_j}{(s+\delta_j)^2}>0
\qquad(s\ne-\delta_1,-\delta_2).
\]
Both \(W\) and \(W+2\rho\) increase from \(-\infty\) to \(+\infty\)
on each of \((-\infty,-\delta_2)\) and \((-\delta_2,-\delta_1)\),
and from \(-\infty\) to a positive value on \((-\delta_1,0]\).
Each therefore has exactly one simple zero in each of these three intervals.
There is no cancellation with \(\Delta\), since
\(N(-\delta_1)=-\zeta_1\delta_1(\delta_2-\delta_1)\ne0\) and
\(N(-\delta_2)=\zeta_2\delta_2(\delta_2-\delta_1)\ne0\).
The same identities hold with \(N\) replaced by
\(N+2\rho\Delta\), because \(\Delta(-\delta_j)=0\).  Thus no root is
cancelled in either quotient.  Write the roots of \(N\) as
\(-\lambda_1,-\lambda_2,-\lambda_3\), and those of
\(N+2\rho\Delta\) as
\(-\widetilde\lambda_1,-\widetilde\lambda_2,-\widetilde\lambda_3\), where
all \(\lambda_i\) and \(\widetilde\lambda_i\) are positive.  The two functions
then have the partial-fraction expansions
\[
\frac1{W(s)}
=
\sum_{i=1}^3\frac{a_i}{s+\lambda_i},
\qquad
\frac1{2\rho+W(s)}
=
\sum_{i=1}^3\frac{\widetilde a_i}{s+\widetilde\lambda_i}.
\]
Arb computes outward-rounded interval enclosures for these pole parameters and
residues.

The verifier expands the radial Hermite factors using the Laguerre convention
\[
L_j^{(\alpha)}(s)
=
\sum_{k=0}^j(-1)^k
\binom{j+\alpha}{j-k}\frac{s^k}{k!},
\]
under which
\[
e_{a,b}(r)
=
\begin{cases}
(-1)^b\sqrt{\dfrac{b!}{a!}}\,
r^{a-b}L_b^{(a-b)}(r^2),&a\ge b,\\[7pt]
(-1)^a\sqrt{\dfrac{a!}{b!}}\,
r^{b-a}L_a^{(b-a)}(r^2),&b>a.
\end{cases}
\] 

For \(k\in\mathbb N_0\) and \(\lambda>0\), define
\(J_k(\lambda)=\int_0^\infty s^ke^{-s}(s+\lambda)^{-1}\dd s\).
For example, the second partial-fraction expansion gives
\[
R
=
\int_0^\infty
\frac{s e^{-s}}{2\rho+W(s)}\dd s
=
\sum_{i=1}^3\widetilde a_iJ_1(\widetilde\lambda_i).
\]
More generally, the Laguerre formulas above turn every numerator in
\eqref{eq:residual-matrix} and \eqref{eq:woodbury-data} into a polynomial
\(P(s)=\sum_k c_ks^k\).  Substituting the first partial-fraction expansion,
for instance, gives
\[
\int_0^\infty\frac{P(s)e^{-s}}{W(s)}\dd s
=
\sum_{i=1}^3\sum_k a_ic_kJ_k(\lambda_i),
\]
and the second expansion gives the same formula with tildes.  Thus every
required matrix entry is a finite linear combination of the moments
\(J_k(\lambda)\).
They are evaluated from
\(J_0(\lambda)=e^\lambda E_1(\lambda)\) and
\(J_{k+1}(\lambda)=k!-\lambda J_k(\lambda)\), where
\(E_1(\lambda)=\int_\lambda^\infty e^{-t}t^{-1}\dd t\).
The recurrence follows from
\(
s^{k+1}/(s+\lambda)=s^k-\lambda s^k/(s+\lambda)
\).
The normalized radial Hermite products are generated with exact factorial and
binomial coefficients.  Their degrees are at most \(21\) in \(s\), and
\(\phi_{10}^2=s(L_{10}^{(1)}(s))^2/11\) already has degree \(21\); hence moments
through degree \(21\) contain every required entry.

\subsubsection{The Gram inequalities}

The verifier forms the matrices
\(I-\mathcal M^{p,q}(W)\) for
\(0\le p\le11\) and \(0\le q\le10\),
giving \(132\) nonempty blocks and \(570\) pivots in total.  Within each block
the indices \(\ell\in\mathcal I_{p,q}^{\Lambda_*}\) are ordered increasingly, and the
verifier performs an interval \(LDL^*\) decomposition without pivoting.  A
strictly positive lower endpoint for every pivot proves
\eqref{eq:residual-condition}; the smallest endpoint is the quantity in
\eqref{eq:residual-pivot-bound}, attained in the \((2,1)\) block.

\subsubsection{The Schur condition and the ratio}

The verifier substitutes the moment enclosures from
Section~\ref{app:lb-integrals} into the formulas for
the scalar \(R\), the vector
\(\boldsymbol{b}\), and the matrix \(G\).  With
\(B=\operatorname{diag}(\beta_1,\ldots,\beta_{10})\), an interval \(LDL^*\)
decomposition certifies that the exact matrix \(B^{-1}+G\) is positive
definite. Hence \((B^{-1}+G)x=\boldsymbol b\)
has a unique solution, which the verifier encloses.  It then evaluates
\(S=R-\boldsymbol{b}^Tx\) with outward rounding.  The positive lower endpoint
for \(1-S\), recorded in the proof of
Proposition~\ref{prop:directed-certificate}, proves \(S<1\).

It remains to evaluate the weight integral \(\mu_W\).  Using
\(\int_0^\infty se^{-s}\dd s=1\) and
\(J_0(\delta)=e^\delta E_1(\delta)\), the definition of \(W\) gives
\begin{equation*}
\mu_W
=
\int_0^\infty W(s)e^{-s}\dd s
=
\alpha+\eta
+\sum_{j=1}^2
\zeta_j\bigl(1-\delta_j e^{\delta_j}E_1(\delta_j)\bigr).
\end{equation*}
Outward-rounded evaluation of this formula gives the bound on \(\mu_W\)
in \eqref{eq:mu-upper}.

For the ratio, put \({U=0.716627735168304}\) and
\({L=1.35584631827168}\). Exact rational arithmetic with the stated
value of \(\rho\) gives
\(1-\rho-L(\rho+U)={6.85709570393239744\times10^{-15}}>0\).
Since \(\mu_W<U\), it follows that
\((1-\rho)/(\rho+\mu_W)>(1-\rho)/(\rho+U)>L\).
All displayed decimal parameters are interpreted exactly as rational
numbers. Arb, through python-flint \(0.8.0\), evaluates the roots,
integrals, and matrix quantities using outward-rounded ball arithmetic
at 180-digit precision.

\subsection{The pointwise bound}
\label{app:dual-verification}

To find the dual parameters, we fixed 180 logarithmically spaced rates
\(q_j/(1-q_j)\) between \(10^{-3}\) and \(10^2\). We then numerically solved
a second-order cone program for nonnegative weights \(c_j\) with
\(\sum_jc_j=1\), minimizing the maximum of
\(\sqrt{X(s)^2+Y(s)^2}\) over 1801 sampled values of \(s\).
The resulting decimals were interpreted as exact rational numbers, with
the last weight adjusted so that the weights sum to exactly one.

We first check in exact arithmetic that the 180 input pairs satisfy
\(c_j>0\), \(0<q_j<1\), and \(\sum_jc_j=1\).  Lemma~\ref{lem:atomic-matching-pair}
then gives the coefficient matching.  The remaining task is to prove the
uniform pointwise bound
\(X(s)^2+Y(s)^2<1.35584697425050\) for every \(s\ge0\).
We verify the bound on \([0,10^5]\) and estimate the remaining tail
analytically.

\subsubsection{Regularity at the origin}

To apply Taylor estimates at the origin, we write
\(Y(s)=\sqrt{s}H(s)\) and work with
\(\mathcal F(s)=X(s)^2+sH(s)^2\).

For the \(180\) exact pairs specified at the start of
Appendix~\ref{app:verification}, put
\(q_*=\max_jq_j<1\) and define
\[
\kappa_j(v)=\frac{q_je^{-v^2}}{1-q_je^{-v^2}},
\qquad
\omega_j(v)=\frac{e^{-v^2}}{(1-q_je^{-v^2})^2}.
\]
Equation \eqref{eq:Y-formula} gives
\begin{equation*}
H(s)
=
\frac2{\sqrt\pi}
\sum_{j=1}^{180}c_j
\int_0^\infty \omega_j(v)e^{-s\kappa_j(v)}\dd v.
\end{equation*}
The functions \(\kappa_j\) are uniformly bounded and the \(\omega_j\) are
integrable.  Consequently, \(H\) extends to an entire function of \(s\),
with
\begin{equation*}
H^{(k)}(s)
=
\frac{2(-1)^k}{\sqrt\pi}
\sum_{j=1}^{180}c_j
\int_0^\infty
\omega_j(v)\kappa_j(v)^ke^{-s\kappa_j(v)}\dd v.
\end{equation*}
For \(s\ge0\), each
\((-1)^kH^{(k)}(s)\) is nonnegative and nonincreasing.  Likewise,
\(-X\) is a finite positive mixture of exponentials, so
\(\abs{X^{(k)}(s)}\) is nonincreasing for every \(k\ge0\).
Moreover, \(H(0)>0\), so \(Y\) is not \(C^1\) at the origin.

\subsubsection{Verification on \texorpdfstring{\(0\le s\le10^5\)}{0 <= s <= 100000}}

We prove \(\mathcal F(s)<1.35584697425050\) for every \(0\le s\le10^5\).
An Arb calculation at \(80\)-digit precision covers this range by \(98\)
adjacent intervals.  At each left endpoint \(a\), derivatives are enclosed
by quadrature with explicit remainder bounds.
Set
\[
g_{k,a}(v)
=
\sum_{j=1}^{180}
c_j\omega_j(v)\kappa_j(v)^k e^{-a\kappa_j(v)}
\qquad(0\le k\le8).
\]
\begin{samepage} 
The verifier divides \([0,8]\) into \(600\) Simpson panels: \(200\) panels
of width \(1/200\) on \([0,1]\), \(200\) of width \(1/100\) on
\([1,3]\), and \(200\) of width \(1/40\) on \([3,8]\).  On a panel
\(J=[u,u+h]\), it uses the validated remainder
\begin{equation}
\left|
\int_J g_{k,a}(v)\dd v
-\frac h6
\left(
g_{k,a}(u)+4g_{k,a}(u+h/2)+g_{k,a}(u+h)
\right)
\right|
\le
\frac{h^5}{2880}
\sup_{v\in J}\abs{\partial_v^4g_{k,a}(v)}.
\label{eq:validated-simpson}
\end{equation}
\end{samepage}
On each panel \(J\), the verifier evaluates \(g_{k,a}\) using truncated
power-series arithmetic with interval coefficients.
The input series in the formal variable \(x\) has constant coefficient
\(J\), linear coefficient one, and all higher coefficients zero.
The resulting coefficient of \(x^4\) encloses
\(\partial_v^4g_{k,a}(v)/4!\) for every \(v\in J\).
Multiplying the maximum of the absolute values of its endpoints by
\(4!\) gives a uniform bound for the fourth derivative on \(J\), as required in
\eqref{eq:validated-simpson}.
The remaining integral satisfies
\begin{equation}
\int_8^\infty g_{k,a}(v)\dd v
\le
\frac{
q_*^k e^{-64(k+1)}
}{
16(k+1)(1-q_*e^{-64})^{k+2}
},
\label{eq:quadrature-tail}
\end{equation}
obtained by dropping \(e^{-a\kappa_j(v)}\), using \(\sum_jc_j=1\), and
applying the Gaussian-tail bound
\[
\int_8^\infty e^{-(k+1)v^2}\dd v
\le \frac{e^{-64(k+1)}}{16(k+1)}.
\]
Scaled by \(2/\sqrt\pi\), the panel sums, Simpson remainders
and tail \eqref{eq:quadrature-tail} enclose
\((-1)^kH^{(k)}(a)\).  The derivatives of \(X\) are enclosed directly from
\eqref{eq:X-formula}.

Since the derivative magnitudes of \(H\) and \(X\) are nonincreasing,
their enclosures at \(a\) give upper bounds throughout \([a,b]\).
The verifier
forms the degree-seven Taylor polynomial of \(\mathcal F\) at \(a\) and
converts it to Bernstein form on \([a,b]\).
After rescaling this interval to \([0,1]\), the Bernstein basis functions
are nonnegative and sum to one.
The largest upper endpoint of the interval coefficients therefore bounds
the polynomial throughout \([a,b]\).
The remainder is controlled by
\begin{equation*}
\left|
\mathcal F(s)-\sum_{j=0}^7\frac{\mathcal F^{(j)}(a)}{j!}(s-a)^j
\right|
\le
\frac{(b-a)^8}{8!}
\sup_{u\in[a,b]}\abs{\mathcal F^{(8)}(u)},
\qquad s\in[a,b],
\end{equation*}
together with
\begin{equation*}
\abs{\mathcal F^{(8)}(s)}
\le
\abs{(X^2)^{(8)}(s)}
+s\abs{(H^2)^{(8)}(s)}
+8\abs{(H^2)^{(7)}(s)}.
\end{equation*}
The product derivatives are bounded by Leibniz's rule, and \(s\le b\).
Rounding the largest certified upper bound over the \(98\) intervals
outward gives
\begin{equation}
\sup_{0\le s\le10^5}\mathcal F(s)
\le C_{\mathrm{cert}}:=1.355846974250492.
\label{eq:finite-cover-bound}
\end{equation} 

\subsubsection{Verification for \texorpdfstring{\(s\ge10^5\)}{s >= 100000}}

For \(s\ge10^5\), the inequality
\(
e^{-z}\le(k/e)^kz^{-k}
\)
with \(k=3/4\), applied inside \eqref{eq:Y-formula}, gives
\begin{equation*}
0\le Y(s)
\le
\frac{(k/e)^k}{\sqrt{1-k}}s^{1/2-k}
\sum_{j=1}^{180}c_jq_j^{-k}(1-q_j)^{k-2},
\qquad k=\frac34.
\end{equation*}
Its right-hand side is decreasing, as is \(\abs{X(s)}\).  Outward-rounded
evaluation at \(s=10^5\) gives
\(Y(s)<0.488503616157382\) and \(\abs{X(s)}<2.878\times10^{-50}\),
uniformly on the tail.  Thus \(X(s)^2+Y(s)^2<0.238636\) throughout the tail.
Since \(0.238636<C_{\mathrm{cert}}\), \eqref{eq:finite-cover-bound} gives
the same bound \(C_{\mathrm{cert}}\) for the supremum over the whole
half-line. The separation from the claimed threshold is
\(1.35584697425050-C_{\mathrm{cert}}=8\times10^{-15}>0\),
which proves \eqref{eq:dual-cover-bound}.

\bibliographystyle{alpha}
\bibliography{references}

\end{document}